\documentclass[11pt,leqno,a4]{amsart}
\usepackage{amssymb,amsmath,amscd,graphicx,amsthm,mathrsfs}
\usepackage[pdftex,bookmarks,colorlinks,breaklinks]{hyperref}
\hypersetup{linkcolor=blue,citecolor=red}
\usepackage{comment}
\usepackage{float}
\usepackage{multicol}

\usepackage[top=32mm,bottom=32mm,right=20mm,left=20mm]{geometry}

\theoremstyle{plain}
\newtheorem{theo}{Theorem}

\newtheorem{lem}{Lemma}
\newtheorem{prop}{Proposition}
\newtheorem{cor}{Corollary}
\theoremstyle{definition}
\newtheorem{remark}{Remark}

\theoremstyle{remark}

\newcommand{\F}{{\mathbb{F}}}
\newcommand{\R}{\mathcal{O}}
\newcommand{\PP}{\mathfrak{p}}
\newcommand{\co}{\widetilde}
\newcommand{\m}{\mathbf}
\newcommand{\cm}{\mathtt}
\newcommand{\SL}{G^1}
\newcommand{\GL}{G}
\newcommand{\TG}{T}
\newcommand{\TS}{T^1}
\newcommand{\BG}{B}
\newcommand{\BS}{B^1}
\newcommand{\UNG}{N}
\newcommand{\UNS}{N^{1}}
\newcommand{\KS}{K^1}
\newcommand{\KG}{K}
\newcommand{\TSL}{\widetilde{T^1}}
\newcommand{\TGL}{\co{\TG}}
\newcommand{\BSL}{\co{B^1}}
\newcommand{\BGL}{\co{B}}
\newcommand{\KSL}{\co{K^1}}
\newcommand{\KGL}{\co{K}}
\newcommand{\AS}{A^1} 
\newcommand{\AG}{A}
\newcommand{\chis}{\chi}
\newcommand{\chig}{\chi'}
\newcommand{\pis}{\rho}
\newcommand{\pig}{\rho'}
\newcommand{\psrg}{\pi'}
\newcommand{\ress}{\mathrm{Res}_{\KSL}\mathrm{Ind}_{\BSL}^{\co\SL}\rho} 
\newcommand{\resg}{\mathrm{Res}_{\KGL}(\mathrm{Ind}_{\BGL}^{\co\GL}\pig)} 
\newcommand{\Res}{\mathrm{Res}}
\newcommand{\Ind}{\mathrm{Ind}}
\newcommand{\Hom}{\mathrm{Hom}}
\newcommand{\KSec}{\co{K}_0}

\newcommand{\w}{\co{w}}
\newcommand{\val}{\mathrm{val}}

\newcommand{\ddd}{\mathrm{dg}}
\newcommand{\lt}{\mathrm{lt}}
\newcommand{\n}{\underline{n}}
\newcommand{\done}[1]{\iota(#1)}
\newcommand{\varpichar}{\vartheta}

\newcommand {\V}[1]{\co V_{{#1}}}

\title{Branching Rules for $n$-fold Covering Groups of $\mathrm{SL}_2$ over a Non-Archimedean Local Field}
\author[C. Karimianpour]{Camelia Karimianpour}
\address{Camelia Karimianpour, Department of Mathematics and Statistics, University of Ottawa, 585 King Edward, Ottawa, ON K1N
	6N5, Canada.}
\email{ckari099@uottawa.ca}
\begin{document}
\begin{abstract}
Let $\co \SL$ be the $n$-fold covering group of the special linear group of degree two, over a non-Archimedean local field. We determine the decomposition into irreducibles of the restriction of the principal series representations of $\co \SL$ to a maximal compact subgroup of $\co\SL$.
\end{abstract}
\let\thefootnote\relax
\let\thefootnote\relax
\maketitle
\section{Introduction}

In this paper, covering groups, also known in the literature as metaplectic groups, are central extensions of a simply connected simple and split algebraic group, over a non-Archimedean local field $\F$, 
by the group of the $n$-th roots of unity, $\mu_n$. The problem of determining this class of groups was studied by Steinberg~\cite{Steinberg} and Moore~\cite{Moore} in $1968$, and further completed by Matsumoto~\cite{Matsumoto} in $1969$ for simply connected Chevalley groups. Around the same time, Kubota independently constructed $n$-fold covering groups of $\mathrm{SL}_2$~\cite{kubota} and $\mathrm{GL}_2$~\cite{kubota2}, by means of presenting an explicit $2$-cocycle. Kubota's cocycle is expressed in terms of the $n$-th Hilbert symbol.

Since then, there have been a number of studies of representations of this class of groups from different perspectives, among them being the work of H. Arit\"urk~\cite{Ariturk}, D. A. Kazhdan and S. J. Patterson~\cite{Kazhdan-Patterson}, C. Moen~\cite{Moen},  D. Joyner~\cite{JoynerII,JoynerI}, G. Savin~\cite{Savin}, M. Weissman and T. Howard~\cite{Weissman-Howard}, and P. J. McNamara~\cite{mcnamara}. 

In this paper, we consider the principal series representations of the $n$-fold covering group $\widetilde{\mathrm{SL}}_2(\F)$ of $\mathrm{SL}_2(\F)$. The principal series representations of $\widetilde{\mathrm{SL}}_2(\F)$ are those representations that are induced from the inverse image $\BSL$ of a Borel subgroup $\BS$ of $\mathrm{SL}_2(\F)$. The construction of those representations of $\BSL$ that are trivial on the unipotent radical of $\BSL$ brings us to the study of the irreducible representations of the metaplectic torus $\TSL$, i.e., the inverse image of the split torus, $\TS$, of $\mathrm{SL}_2(\F)$ in $\widetilde{\mathrm{SL}}_2(\F)$.

An important feature of $\TSL$, which differentiates the nature of its representations from those of a linear torus, is that it is not abelian. However, it is a Heisenberg group and its irreducible representations are governed by the Stone-von Neumann theorem. The Stone-von Neumann theorem characterizes irreducible representations of Heisenberg groups, according to their central characters. Indeed, given a character of the centre of a Heisenberg group that satisfies some mild conditions, the Stone-von Neumann theorem provides a recipe to construct the corresponding, unique up to isomorphism,  irreducible representation of the Heisenberg group. The construction involves induction from a maximal abelian subgroup of the Heisenberg group. We only consider those characters of the centre of $\TSL$ where $\mu_n$ acts by a fixed faithful character.

Once an irreducible representation $\rho_{\chi}$, with central character $\chi$, of $\TSL$ is obtained, the principal series representation $\pi_{\chi}$ of $\widetilde{\mathrm{SL}}_2(\F)$ is $\Ind_{\BSL}^{\widetilde{\mathrm{SL}}_2(\F)}\rho_\chi$, where $\rho_\chi$ is trivially extended on the unipotent radical subgroup  of $\BSL$. These representations admit several open questions. The question we consider, and answer, in this paper is to decompose $\pi_{\chi}$ upon the restriction to the inverse image $\KSL$ of a maximal compact subgroup $\KS$ of $\mathrm{SL}_2(\F)$. We refer to this decomposition as the $\mathsf{K}$-type decomposition. We assume $n|q-1$, where $q$ is the size of the residue field of $\F$, so that the central extension $\co {\mathrm{SL}_2}(\F)$ splits over $\KS$.

The study the decomposition of the restriction of representations to a particular subgroup is a common technique in representation theory. In the theory of real Lie groups, restriction to maximal compact subgroups retains a lot of information from the representation; in fact, such a restriction is a key step towards classifying irreducible unitary representations. In the case of reductive groups over $p$-adic fields, investigating the decomposition upon restriction to maximal compact subgroups reveals a finer structure of the representation, in the interests of recovering essential information about the original representation.

The $\mathsf{K}$-type problem for reductive $p$-adic groups is visited and solved in certain cases, including the principal series representations of $\mathrm{GL}(3)$~\cite{Campbell-NevinsUnrm,Campbell-Nevinsram,Onn-Singla}, and $\mathrm{SL}(2)$~\cite{monica,NevinsPatterns}, representations of $\mathrm{GL}(2)$~\cite{CasselmanKtype}, and supercuspidal representations of $\mathrm{SL}(2)$~\cite{Nevins-supercusp}.

The main idea is to reduce the problem to calculating the dimensions of certain finite-dimensional Hecke algebras. The key calculation for determining the decomposition is the determination of certain double cosets that support intertwining operators for the restricted principal series representation (Proposition~\ref{heckedim} and Proposition~\ref{multiplicityHeckesl}). 

Our method is aligned with the one in~\cite{monica} for the linear group $\mathrm{SL}_2(\F)$; however, the technicalities in the covering case are much more involved than the linear case, and the results are fairly different. For instance, the $\mathsf{K}$-type decomposition is no longer multiplicity-free (Corollary~\ref{multiplicitysl}).

This paper is organized as follows. In Section~\ref{NotationandBackground}, we present Kubota's construction of the covering group of $\mathrm{SL}_2(\F)$, in Section~\ref{StructureTheory} we overview the structure of this covering group and compute some subgroups of our interest. We compute the $\mathsf{K}$-type decomposition for the principal series representations of $\co{\mathrm{SL}_2}(\F)$ in Section~\ref{branching-rule-section}. This decomposition is completed by considering a similar problem for the $n$-fold covering group of $\mathrm{GL}_2(\F)$ in Section~\ref{branching-rule-sectionGL}. Our main result, Theorem~\ref{MainTheoremofK-types}, is stated in Section~\ref{mainresultsection}.

\section{Notation and Background}\label{NotationandBackground}
Let $\F$\label{symF} be a non-Archimedean local field with the ring of integers $\R$\label{symR} and the maximal ideal $\PP$\label{symPP} of $\R$. Let $\kappa:=\R/\PP$\label{symkappa} be the residue field and $q=|\kappa|$\label{symq} be its cardinality. Let $\R^\times$ denote the group of units in $\R$. We fix a unoformizing element $\varpi$\label{symvarpi} of $\PP$. For every $x\in \F^{\times}$, the valuation of $x$ is denoted by $\val(x)$, and $|x|=q^{-\val(x)}$. Let $n\geq 2$ be an integer such that $n|q-1$. Set $\n=n$ if $n$ is odd, and $\n=\frac{n}{2}$ if $n$ is even. We assume that $\F$ contains the group $\mu_n$ of $n$-th roots of unity. 
	
Set $\GL=\mathrm{GL}_2(\F)$, and $\SL=\mathrm{SL}_2(\F)$. Let $\BS$ ($\BG$) be the standard Borel subgroup of $\SL$ $(\GL)$ and $\UNS$ ($\UNG$) be its unipotent radical, and let $\TS$ $(\TG)$ be the standard torus in $\SL$ $(\GL)$. Set $\KS=\mathrm{SL}_2(\R)$ ($\KG=\mathrm{GL}_2(\R)$) to be a maximal compact subgroup of $\SL$ ($\GL$). By the Iwasawa decomposition, we have $\SL=\TS \UNS \KS$ ($\GL=\TG \UNG \KG$). Our object of study is the central extension $\co \SL$ of $\SL$ by $\mu_n$,
\begin{equation}\label{centralextension}
0\to \mu_n\stackrel{\mathsf{i}}\to \co \SL\stackrel{\mathsf{p}}\to \SL\to 0,
\end{equation}
where $\mathsf{i}$ and $\mathsf{p}$ are natural injection and projection maps respectively. The group $\co\SL$, which we call the $n$-fold covering group of $\SL$, is constructed explicitly by Kubota~\cite{kubota}. In order to describe Kubota's construction, we need knowledge of the $n$-th Hilbert symbol $(\:,\:)_n: \F^{\times}\times\F^{\times}\to \mu_n$. Under our assumption on $n$, the $n$-th Hilbert symbol is given via  $(a,b)_n=\overline {c}^{\frac{q-1}{n}}$, where $c=(-1)^{\val(a)\val(b)}\frac{{a}^{\val(b)}}{b^{\val(a)}}$, and $\overline{c}$ is the image of $c$ in $\kappa^{\times}$.
We benefit from the properties of the $n$-th Hilbert symbol, which can be found in~\cite[Ch XIV]{localfieldSerre}. In particular, we benefit extensively from the following fact: $(a,b)_n=1$ for all $a\in \F^{\times}$, if and only if $ b\in {\F^{\times}}^n$.  

Define the map $\beta:\SL\times{\SL}\rightarrow{\mu_n}$ by
\begin{eqnarray}\label{eqbeta}
\beta(\m{g}_1,\m{g}_2)=\left(\frac{X(\m{g}_1\m{g}_2)}{X(\m{g}_1)},\frac{X(\m{g}_1\m{g}_2)}{X(\m{g}_2)}\right)_n,\text{ where }X\left(\left(\begin{array}{cc}a&b\\
c&d\end{array}\right)\right)=\begin{cases}c&\quad\text{if $c\neq{0}$}\\
d&\quad\text{otherwise.}
\end{cases} 
\end{eqnarray}
In~\cite{kubota} Kubota proved that $\beta$ is a non-trivial $2$-cocycle in the continuous second cohomology group of $\SL$ with coefficients in $\mu_n$; whence, $\co\SL=\SL\times \mu_n$ as a set, with the multiplication given via $(\m g_1, \zeta_1)(\m g_2, \zeta_2)=(\m g_1\m g_2,\beta(\m g_1,\m g_2) \zeta_1\zeta_2)$, for all $\m g_1, \m g_2\in \SL$ and $\zeta_1,\zeta_2\in \mu_n$. 

In 1969, Kubota extends the map $\beta$ to a $2$-cocycle $\beta'$ for $\co \GL$ in~\cite{kubota2}, which defines the $n$-fold covering group $\co \GL\cong \F^{\times}\ltimes\co \SL$ of $\GL$. The covering group $\co \GL$ fits into the exact sequence $0\to \mu_n\stackrel{\mathsf{i}}\to \co \GL\stackrel{\mathsf{p}}\to \GL\to 0$.

For all $t, s\in \F^{\times}$, set $\ddd(t)=\left(\begin{smallmatrix}
t& 0\\
0&t^{-1}
\end{smallmatrix}\right)\in \TS$, $\ddd{(t, s)}= \left(\begin{smallmatrix}
t& 0\\
0&s
\end{smallmatrix}\right)\in \TG$, and $\done{t}=\left(\ddd(t), 1\right)\in \TSL$. Set $w=\left(\begin{smallmatrix}
0 & 1\\
-1& 0\\
\end{smallmatrix}
\right)$, and $\w=(w,1)\in \co\SL$. Moreover, for matrices $X$ and $Y$, with $Y$ invertible, let $X^Y:=Y^{-1}XY$ and ${}^YX:=YXY^{-1}$ denote the conjugations of $X$ by $Y$.

\section{Structure Theory}\label{StructureTheory}
For any subgroup $H$ of $\SL$, the inverse image $\co H:=\mathsf{p}^{-1}(H)$ is a subgroup of $\co \SL$. In particular, we are interested in the subgroups $\TSL$, $\BSL$, and $\KSL$ of $\co\SL$. We say the central extension splits over the subgroup $H$ of $\SL$, if there exists an isomorphism that yields $\mathsf{p}(H)^{-1}\cong H\times \mu_n$. 

It is not difficult to see that $\TSL$ is not commutative, and hence, the central extension does not split over $\TS$ (and therefore neither over $\BS$). Additionally, it is easy to see that the commutator subgroup $[\TSL, \TSL]\cong \mu_{\n}$ is central in~\eqref{centralextension}; which implies that $\TSL$ is a two-step nilpotent group, also known as a Heisenberg group. Clearly, $\mu_n\in Z(\TSL)$, indeed, using the properties of the Hilbert symbol and some elementary calculation, one can show that $Z(\TSL)=\{({\ddd(t)}
,\zeta)\mid t\in{{\F^\times}^{\n}},\zeta\in\mu_n\}.$ 
\begin{lem}\label{indexofcenter}
	The index of $Z(\TSL)$ in $\TSL$ is $\n^2$.
\end{lem}
\begin{proof}
Note that $[\TSL:Z(\TSL)]=[\F^{\times}:{\F^{\times}}^{\n}]$, which because ${\F^{\times}}\cong\mathcal{O}^{\times}\times\mathbb{Z}$, is equal to $\n[\mathcal{O}^{\times}:{\mathcal{O}^{\times}}^{\n}]$.  Consider the homomorphism $\phi:\mathcal{O}^{\times}\rightarrow{\mathcal{O}^{\times}}^{\n}$. Then $\ker(\phi)=\{x\in{\mathcal{O}^\times}|\:x^{\n}=1\}$. Note that $f(x)=x^{\n}-1=0$ has $(\n,q-1)$, which equals $\n$ under our assumption of $n|q-1$, solutions in the cyclic group $\kappa^\times$. By Hensel's lemma, any such root in $\kappa^{\times}$ lifts uniquely to a root in $\R^{\times}$. It follows that, $|\ker(\phi)|=\n$. Therefore,  $[\mathcal{O}^{\times}:{\mathcal{O}^{\times}}^{\n}]=|\ker(\phi)|=\n$, and the result follows.
\end{proof}
In order to construct principal series representations of $\co \SL$ in Section~\ref{branching-rule-section}, we need to construct irreducible representations of the Heisenberg group $\TSL$. To do so, we need to identify a maximal abelian subgroup of $\TSL$. Set $\AS=C_{\TSL}(\TSL\cap \KSL)$, to be the centralizer of $\TSL\cap \KSL$ in $\TSL$. It is not difficult to calculate that $\AS=\{(\ddd(a),\zeta)\mid a\in \F^{\times},\:\n|\val(a),\:\zeta\in\mu_n\}$, and see that it is abelian. Observe that $\TSL\cap \KSL\subset \AS$ implies that $\AS$ is a maximal abelian subgroup. Note that $[\TSL: \AS]=[\mathbb{Z}:\n\mathbb{Z}]=\n$.

Let $\UNS$ be the unipotent radical of $\BS$. It follows directly from the Kubota's formula for $\beta$ that $\beta|_{\UNS}$ is trivial, so $\UNS\times \{1\}$ is a subgroup of $\co \SL$. We identify $\UNS$ with $\UNS\times \{1\}$. Under this identification, we have the covering analogue of the Levi decomposition: $\BSL=\TSL\ltimes{\UNS}$.   

Next, we describe a family of compact open subgroups of $\co \SL$. It is proven in~\cite{kubota2} that 
\begin{equation}\label{KSeciso}
\KSL\to \KS\times\mu_n,\quad (\m k,\zeta)\mapsto(\m k,s(\m k)\zeta),\text{ where } s\left(\left(
\begin{array}{cc}
a & b \\
c & d \\
\end{array}
\right)\right)=\begin{cases}(c,d)_n,&0<\val(c)<\infty\\
1,&\text{otherwise.}
\end{cases}
\end{equation}
is an isomorphism. The image of $\KS$ in $\KSL$ under the isomorphism~\eqref{KSeciso} is the subgroup \(
\KSec:=\{(\m k,s(\m k)^{-1})\mid \m k\in \KS\}
\)
of $\KSL$. Consider the compact open congruent subgroups 
\(
\KS_j:=\{\m g\in \KS \mid \m g\equiv \mathrm{I}_{2}\mod \PP^j\}\), for $j\geq 1$, of $\KS$. 
\begin{lem}\label{K_i}
	The central extension~\eqref{centralextension} splits trivially over each of the subgroups $\KS_j$, $j\geq{1}$, $\TS\cap \KS$, and $\BS\cap \KS$.
\end{lem}
\begin{proof}
Using the Hensel's lemma, it is easy to see that $1+\PP\subset {\R^{\times}}^n$. Then, it follows from~\eqref{KSeciso} and properties of the $n$-th Hilbert symbol that, for all $i\geq 1$, $s|_{\KS_j}$ is trivial. On the other hand, it follows directly from~\eqref{KSeciso} that  $s|_{\TS \cap \KS}$ and $s|_{\BS \cap \KS}$ are trivial.
\end{proof}
We identify $\KS_j\cong \KS_j\times\{1\}$, $j\geq{1}$, ${\BS}\cap \KS\cong{(\BS\cap{\KS})}\times \{1\}$ and $\TS\cap{\KS}\cong{(\TS\cap{\KS})\times\{1\}}$ as subgroups of $\KSL$.

In a similar way, we define the subgroups $\TGL$, $\BGL$  and $\KGL$ of $\co \GL$ to be the inverse images of the standard torus, Borel, and the maximal compact $\KG=\mathrm{GL}(\R)$ subgroups of $\GL$ respectively. The central extension $\co \GL$ does not split over $\TG$. Moreover, $\TGL$ is a Heisenberg group. It is not difficult to see that $Z(\TGL)=\{(\ddd(s,t),\zeta)\mid s,t\in {\F^{\times}}^{n}, \zeta\in \mu_n\}$, and $[\TGL: Z(\TGL)]=n^4$. Moreover, set $\AG=C_{\TGL}(\TGL\cap \KGL)=\{ (\ddd(s,t) ,\zeta )\mid s,t\in \F^{\times}, n| \val(s), n|\val(t), \zeta\in \mu_n\}$. Then, $\AG$ is a maximal abelian subgroup of $\TGL$ and $[\TGL: \AG]=n^2$. In addition, $\beta'|_{\UNG}$ is trivial, where $\UNG$ is the unipotent radical of $\BG$. Hence, we can identify $\UNG$ with $\UNG\times \{1\}$. Under this identification, we have the Levi decomposition: $\BGL=\TGL\ltimes{\UNG}$. It is shown in~\cite{kubota2} that the central extension $\co\GL$ splits over $\KG$. For $j\geq 1$, let $\KG_j$ denote the family of compact open congruent subgroups $\{\m g\in \KG \mid \m g\equiv \mathrm{I}_{2}\mod \PP^j\}$ of $\KG$. Similar to Lemma~\ref{K_i}, one can show that $\co\GL$ splits over $\KG_j$, $\TG\cap \KG$ and $\BG\cap \KG$. 


\section{Branching Rules for $\co\SL$} \label{branching-rule-section}
First, we present the construction of the principal series representations of $\co\SL$ following~\cite{mcnamara}. Fix a faithful character $\epsilon : \mu_n\to \mathbb{C}^{\times}$. A representation of $\co \SL$ is genuine if the central subgroup $\mu_n$ acts by $\epsilon$. Such representations do not factor through representations of $\SL$. The construction of principal series representations of $\co \SL$ is based on the essential fact that $\TSL$ is a Heisenberg subgroup, and hence its representations are governed by the Stone-von Neumann theorem, which we state here. See~\cite{mcnamara} for the proof. 
\begin{theo}[Stone-von Neumann]\label{StoneVonNeumman}
	Let $H$ be a Heisenberg group with center $Z(H)$ such that $H/Z(H)$ is finite, and let $\chi$ be a character of $Z(H)$. Suppose that $\ker(\chi)\cap[H,H]=\{1\}$. Then there is a unique (up to isomorphism) irreducible representation $\pi$ of $H$ with central character $\chi$. Let $\mathsf{A}$ be any maximal abelian subgroup of $H$ and let $\chi_0$ be any extension of $\chi$ to $\mathsf{A}$. Then $\pi\cong \Ind_{\mathsf{A}}^{H}\chi_0$.
\end{theo}
Note that $[\TSL: Z(\TSL)]=\n^2<\infty$. Let $\chis$ be a genuine character of $Z(\TSL)$, so that $\chis|_{\mu_n}=\epsilon$. Thus, $\ker(\chis)\cap [\TSL,\TSL]$ is trivial. Hence Theorem~\ref{StoneVonNeumman} applies: genuine irreducible smooth representations $\pis$ of $\TSL$ are classified by genuine smooth characters of $Z(\TSL)$. Moreover, $\mathrm{dim}(\pis)=[\TSL : \co\AS]=\n$.

Let $\chis_{0}$ be a fixed extension of $\chis$ to $\AS$; so that $(\pis, \Ind_{ \AS}^{\TSL}\chis_0)$ is the unique smooth genuine irreducible  representation of $\TSL$ with central character $\chis$. Let us again write $\pis$ for the genuine smooth irreducible representation of $\TSL$, with central character $\chis$, extended trivially over $\UNS$ to a representation of $\BSL=\TSL\ltimes{\UNS}$. Then the genuine principal series representation of $\co\SL$ associated to $\pis$ is $\Ind_{\BSL}^{\co{\SL}}{\pis}$, where $\Ind$ denotes the smooth (non-normalized) induction. In the rest of this section, we decompose $\Res_{\KSL}\Ind_{\BSL}^{\co{\SL}}\pis$ into irreducible constituents. We drop the adjective ``genuine'' for simplicity.

 Define the character 
 \begin{equation}\label{varpichar}
 \varpichar:\F^{\times}\to\mu_n,\quad
 a\mapsto(\varpi, a)_n.
 \end{equation}
 Observe that $\varpichar$ is ramified of degree one. Set $\varpichar_{{\R^{\times}}^2}:=\varpichar|_{{\R^{\times}}^2}$. Observe that a typical element of $\AS$ can be written as $\left(\ddd(a\varpi^{r\n}),\zeta\right)$, and a typical element of $\TSL\cap \KSL$ can be written as $(\ddd(a),\zeta)$, where $a\in\R^{\times}$, $r\in\mathbb{Z}$, and $\zeta\in \mu_n$.
\begin{lem}\label{distinctonAS}
	Let $\pis$ be the unique irreducible representation of $\TSL$ with central character $\chis$. Then 
	\(\Res_{\AS}\pis\cong\bigoplus_{i=0}^{\n-1}\chis_i,\)
	where the $\chis_i$ are  $\n$ distinct characters of $\AS$ defined by
	\[
	\chis_i\left(\ddd(a\varpi^{\n r}),\zeta\right)=\chis_0\left(\ddd(a\varpi^{\n r}),\varpichar^{2i}(a)\zeta\right),
	\]
	for all $a\in \R^{\times}$, $r\in\mathbb{Z}$, $\zeta\in \mu_n$, and $0\leq i< \n$.
\end{lem}
\begin{proof}
	By Theorem~\ref{StoneVonNeumman}, $\pis\cong \Ind_{\AS}^{\TSL}\chis_0$. By Mackey's theory,
	\(
	\Res_{\AS}\Ind_{\AS}^{\TSL}\chis_0=\bigoplus_{s\in S_{\n}}\Ind_{\AS\cap ^s\!\AS}^{\AS}{\chis_0}^s,
	\)
	where $S_{\n}$ is a complete set of coset representatives for $\AS\backslash\TSL/\AS$. It is not difficult to see that we can choose
	\(
	S_{\n}=\{\left(\ddd(\varpi^i),1\right)\:|\:0\leq i<\n\}.
	\)
	Since $\AS$ is stable under conjugation by $S_{\n}$, $\Ind_{\AS\cap ^s\!\AS}^{\AS}{\chis_0}^s={\chis_0}^s$. Let $\left(\ddd(a\varpi^{r\n}),\zeta\right)\in \AS$, and $s=\left(\ddd(\varpi^i),1\right)\in S_{\n}$. Then
\begin{gather*}
s^{-1}\left(\ddd(a\varpi^{r\n}),\zeta\right)s=\left(\ddd(\varpi^{-i}),(\varpi^i,\varpi^i)_n\right)\left(\ddd(a\varpi^{r\n}),\zeta\right)\left(\ddd(\varpi^i),1\right)\\
=
\left(\ddd(a\varpi^{r\n-i}),(a\varpi^{r\n},\varpi^{-i})_n(\varpi^i,\varpi^i)_n\zeta\right)\left(\ddd(\varpi^i),1\right)
=
\left(\ddd(a\varpi^{r\n}),(\varpi^i,a\varpi^{r\n-i})_n(a\varpi^{r\n},\varpi^{-i})_n(\varpi^i,\varpi^i)_n\zeta\right)\\
=
\left(\ddd(a\varpi^{r\n}),(\varpi,a)_n^{2i}\zeta\right)=\left(\ddd(a\varpi^{r\n}),\varpichar^{2i}(a)\zeta\right).
\end{gather*}
Hence, ${\chis_0}^s\left(\left(\ddd(a\varpi^{r\n}),\zeta\right)\right)=\chis_0\left(\left(\ddd(a\varpi^{r\n}),\varpichar^{2i}(a)\zeta\right)\right)$. Denote this character $\chis_i$. To show that the $\chis_i$, $0\leq i<\n$, are distinct, it is enough to show that $\varpichar^{2i}|_{\R^{\times}}=1$ if and only if $i=0$. Observe that \(\varpichar^{2i}(a)=\overline{a^{-1}}^{\frac{(q-1)2i}{n}}\),
	which is equal to $1$ for all $a\in \R^{\times}$ if and only if $n|2i$. The result follows.
\end{proof}
The characters $\chis_i$ defined in Lemma~\ref{distinctonAS} are clearly distinct when restricted to $\TSL\cap\KSL$ and, again writing $\chis_i$ for these restrictions,  
\begin{eqnarray}\label{TSLcapK}
\Res_{\TSL\cap\KSL}\pis=\bigoplus_{i=0}^{\n-1}\chis_i.
\end{eqnarray}
\begin{prop}\label{ress}
	Let $\chis_i$, $0\leq i< \n$, denote also the trivial extension of the characters in~\eqref{TSLcapK} to $\BSL\cap\KSL$. Then
	$$
	\ress\cong\bigoplus_{i=0}^{\n-1}{\Ind_{\BSL\cap\KSL}^{\KSL}{\chis_i}}.
	$$
	\begin{proof}
		By Mackey's theorem, we have
		\(
		\ress\cong\bigoplus_{x\in{X}}\Ind_{\BSL ^{x^{-1}}\!\!\!\!\!\cap{\KSL}}^{\KSL}\Res_{\BSL^{ x^{-1}}\!\!\!\!\!\cap{\KSL}}\pis^x,
		\)
		where $X$ is a complete set of double coset representatives of $\KSL$ and $\BSL$ in $\co{\SL}$. The Iwasawa decomposition $\KSL\BSL=\co{\SL}$ implies that $X=\{(\mathrm{I}_2,1)\}$ and hence
		\(
		\ress=\Ind_{\BSL\cap{\KSL}}^{\KSL}\Res_{\BSL\cap{\KSL}}\pis.
		\)
		The result follows from~\eqref{TSLcapK}.
	\end{proof}
\end{prop}
Hence, in order to calculate the $\mathsf{K}$-types, it is enough to decompose each ${\Ind_{\BSL\cap\KSL}^{\KSL}{\chis_i}}$, $0\leq i< \n$, into irreducible representations. Note that the induction space $\Ind_{\BSL\cap\KSL}^{\KSL}\chis_i$ is smooth and admissible. 
Fix $i\in\{0,\cdots,\n-1\}$. The smoothness of $ \Ind_{\BSL\cap\KSL}^{\KSL}\chis_i$ implies that
\(
\Ind_{\BSL\cap\KSL}^{\KSL}\chis_i=\bigcup_{l\geq 1}\big(\Ind_{\BSL\cap\KSL}^{\KSL}\chis_i\big)^{\KS_l}.
\)
Note that, by admissibility, $\big(\Ind_{\BSL\cap\KSL}^{\KSL}\chis_i\big)^{\KS_l}$ is finite-dimensional for every $l\geq{1}$ and since $\KS_l$ is normal in $\KSL$, it is $\KSL$-invariant. Hence, to decompose $\Ind_{\BSL\cap\KSL}^{\KSL}\chis_i$ into irreducible constituents, it is enough to decompose each $\big(\Ind_{\BSL\cap\KSL}^{\KSL}\chis_i\big)^{\KS_l}$ into irreducible constituents.

	For any character $\gamma$ of any subgroup $D$ of $\TSL$, we say $\gamma$ is {primitive mod $m$} if $m$ is the smallest strictly positive integer for which $\Res_{D\cap{\KS_{m}}}\gamma=1$. From now on, let $m\geq 1$ be a positive integer such that $\chis$ is primitive mod $m$. Because $1+\PP\subset {\F^{\times}}^n$, $Z(\TSL)\cap{\KS_{m}}=\TSL\cap \KS_m$, for all $m\geq 1$. Note that since $\chis_i|_{Z(\TSL)}=\chis$, $\chis_i|_{\TSL\cap \KS_m}=\chis|_{Z(\TSL)\cap \KS_m}$. Hence, $\chis$ is primitive mod $m$ if and only if the $\chis_i$ for $0\leq i< \n $ are primitive mod $m$. Set $\BSL_l:=(\BSL\cap\KSL)\KS_l$.
\begin{lem}\label{keylammaminor}
	For every $0\leq i <\n$,
	\[
	\big(\Ind_{\BSL\cap{\KSL}}^{\KSL}\chis_i\big)^{{\KS_l}}=\begin{cases}
	\{0\}, &  0<l<{m}\\
	\Ind_{\BSL_l}^{\KSL}\chis_i, & \text{otherwise}.
	\end{cases}	
	\] 
\end{lem}
\begin{proof}
Suppose $0<l<{m}$, and that $f$ is a vector in $	\big(\Ind_{\BSL\cap{\KSL}}^{\KSL}\chis_i\big)^{{\KS_l}}$. Because $\chi_i\big|_{\BSL\cap{\KS_l}}\neq{1}$ for $l<m$, we can choose $\cm b\in{\BSL\cap{\KS_l}}$ such that $\chi_i(\cm b)\neq{1}$. Let $\cm g\in{\KSL}$.  Note that $\KS_l$ is normal in $\KSL$ and hence $\cm g^{-1}\cm b\cm g\in{\KS_l}$. On the one hand, $f(\cm b\cm g)=\chis_i(\cm b)f(\cm g)$; on the other hand,
	\(
	f(\cm b\cm g)=f(\cm g\cm g^{-1}\cm b\cm g)=\left(\cm g^{-1}\cm b\cm g\right)\cdot f(\cm g)=f(\cm g),
	\)
	since $f$ is fixed by $\KS_l$. It follows that $\chis_i(\cm b)f(\cm g)=f(\cm g)$. Our choice of $\cm b$ implies that $f(\cm g)=0$ and because $\cm g$ is arbitrary, $f=0$. However, if $l\geq m$ then $\chis_i|_{\KS_l}=0$ and because $\KS_l$ is normal in $\KSL$, it is not difficult to see that every $\KS_l$-fixed vector $f$ translates on the left by $\BSL_l$ and vice-versa. Hence the result follows. 
\end{proof}
Lemma~\ref{keylammaminor} tells us that, in order to decompose $(\Ind_{\BSL\cap \KSL}^{\KSL}\chis_i)^{\KS_l}$ into irreducible constituents, it is enough to decompose $\Ind_{\BSL_l}^{\KSL}\chis_i$. Hence, we are interested in counting the dimension of \(\Hom_{\KSL} (\Ind_{\BSL_l}^{\KSL}\chis_i,\Ind_{\BSL_l}^{\KSL}\chis_i).\)
By Frobenius reciprocity, this latter space is isomorphic to $\Hom_{\BSL_l}(\Res_{\BSL_l}\Ind_{\BSL_l}^{\KSL}\chis_i، ,\chi_i)$. It follows from Mackey's theory that 
\[
\Res_{\BSL_l}\Ind_{\BSL_l}^{\KSL}\chis_i\cong \bigoplus_{\cm x\in S}\Ind_{\BSL_l^{\cm x^{-1}}\!\!\!\!\!\cap \BSL_l}^{\BSL_l}\chis_i^{\cm x},
\] 
where $S$ is a set of double coset representatives of $\BSL_l\backslash \KSL/\BSL_l$. The set $S$ is a lift to the covering group $\KSL$ of a similar set of double coset representatives calculated in~\cite{monica}. Using the latter set, and because $\mu_n\subset \BSL_l$, it is easy to see that 
\begin{equation}\label{setofdoublecoset}
S=\{(\mathrm{I}_2,1),\co w, \co\lt( x\varpi^r)\mid x\in\{1,\varepsilon\}, 1\leq{r}<l\},  
\end{equation}
 where $\varepsilon$ is a fixed non-square. For $0\leq i,j<\n$, let $\mathcal{H}_{i,j}$ be the Hecke algebra \[\mathcal{H}_{i,j}:=\mathcal{H}(\BSL_l\backslash{\KSL}/\BSL_l,\:\chis_i,\chis_j)=\{f:\KSL\to\mathbb{C}\:|\:f(lgh)=\chis_i(l)f(g)\chis_j(h), \:l,h\in \BSL_l, g\in \KSL\}.\] 
\begin{prop}\label{hecke}
Let $0\leq i,j<\n$. Then $\dim \Hom_{\KSL} (\Ind_{\BSL_l}^{\KSL}\chis_i,\Ind_{\BSL_l}^{\KSL}\chis_j)=\dim \mathcal{H}_{i,j}$.	
\end{prop}
\begin{proof}
On the one hand, observe that $\Hom_{\KSL} (\Ind_{\BSL_l}^{\KSL}\chis_i,\Ind_{\BSL_l}^{\KSL}\chis_j)=\bigoplus_{\cm x\in S}\Hom_{\BSL_l}(\Ind_{\BSL_l^{\cm x^{-1}}\!\!\!\!\!\cap \BSL_l}^{\BSL_l}\chis_i^{\cm x}, \chis_j)$, which by Frobenius reciprocity is equal to
$\bigoplus_{\cm x\in S}\Hom_{\BSL_l^{\cm x^{-1}}\!\!\!\!\!\cap \BSL_l}(\chis_i^\cm x,\chis_j)$. Let $S_{i,j}$ be the set of all $\cm x\in S$ such that $\chis_i(g)=\chis_j(h)$, whenever $h,g\in  \BSL_l$ and $\cm xg\cm x^{-1}=h$. Then $\dim \Hom_{\KSL} (\Ind_{\BSL_l}^{\KSL}\chis_i,\Ind_{\BSL_l}^{\KSL}\chis_j)=|S_{i,j}|$. On the other hand, observe that for every $\cm x\in S$, there exists a function $f\in \mathcal{H}_{i,j}$ with support on the double coset represented by $\cm x$ if and only if $h=\cm xg\cm x^{-1}$ implies $\chis_i(g)=\chis_j(h)$ for all $h,g\in  \BSL_l$. Moreover, the basis of $\mathcal{H}_{i,j}$ is parametrized by such double coset representatives. Hence, $\dim \mathcal{H}_{i,j}=|S_{i,j}|$.
\end{proof}	
Hence, in order to decompose $(\Ind_{\BSL\cap \KSL}^{\KSL}\chis_i)^{\KS_l}$, we are interested in counting the dimension of $\mathcal{H}_{i,i}$. Set $(\TS \cap \KS)^2:=\{\ddd(t^2)\mid t\in \R^{\times}\}$, $\TS_l:=\{\done{t}\mid t\in {\R^{\times}(1+\PP^l})\}$, and $(\TS_l)^2:=\{\done{t^2}\mid t\in {\R^{\times}(1+\PP^l})\}$. It is not difficult to see that $\TS_{l}$ and $(\TS_l)^2$ are subgroups of $(\TSL\cap\KSL)\KS_l$.

\begin{prop}\label{heckedim}
	Let $l\geq m$ and $0\leq i<\n$. Then 
	\(
	\mathrm{dim}\mathcal{H}_{i,i}=
	\begin{cases}
	1+2(l-m),&\text{if}\quad { {\chis_i}|_{(\TS\cap \KS)^2}\neq{1}};\\
	2l,&\text{otherwise.}
	\end{cases}
	\)
\end{prop}
\begin{proof}
	Assume $l\geq m$. Note that $f(\cm b\cm k\cm b')=\chis_i(\cm b)f(\cm k){\chis}_i(\cm b')$ for all $f\in\mathcal{H}_{i,i}$, $\cm b, \cm b'\in \BSL_l$ and $\cm k\in \KSL$. Hence, for every double coset representative $\cm x$ in~\eqref{setofdoublecoset}, there exists a function $f\in\mathcal{H}_{i,i}$, with support on the double coset represented by $\cm x$ if and only if $\cm b\cm x\cm b'=\cm x$ implies that $\chis_i(\cm b\cm b')=1$ for all $\cm b,\cm b'\in \BSL_l$. The set of such double cosets parameterizes a basis for $\mathcal{H}_{i,i}$. We now determine these double cosets. Let
	\(
	\cm b=(\m b,\zeta)=\left(\left(\begin{smallmatrix}t&s\\
	0& t^{-1}\end{smallmatrix}\right),\zeta\right)\), and \(\cm b'=(\m b',\zeta')=\left(\left(\begin{smallmatrix}t'&s'\\
	0& t'^{-1}\end{smallmatrix}\right),\zeta'\right),  
	\)
	where $t,t'\in {\mathcal{O}^{\times}}({1+\PP^l})$, $s,s'\in  \PP^l$ and $\zeta,\zeta'\in\mu_n$ denote arbitrary elements of $\BSL_l$.
	\begin{description}
		\item [The identity coset $\BSL_l$] A function $f\in{\mathcal{H}_{i,i}}$ has support on $\BSL_l$ if and only if $f(\cm b)=\chis_i(\cm b),\:\forall{\cm b\in{\BSL_l}}$. So there is always a function with support on the identity coset, namely $f=\chis_i$.
		
		\item[\textbf{The coset of $\w$}] For $\cm b$ and $\cm b'$ in $\BSL_l$,
		\(
		\cm b\w\cm b'=\w
		\)
		implies, via a quick calculation, that $\m b=\m b'=\ddd(t)$, for some $t\in{\mathcal{O}^{\times}}({1+\PP^l})$ and $\zeta'=\zeta^{-1}$. Therefore, 
		\(
		\chis_i(\cm b\cm b')={\chis}_i\left((\ddd(t),\zeta)(\ddd(t),\zeta^{-1})\right)=\chis_i\left(\ddd(t^2),(t,t)_n\right)
		=\chis_i\left(\ddd(t^2),1\right).
		\)
		So, $\mathcal{H}_{i,i}$ contains a function with support on this coset if and only if $\chis_i(\done{t^2})=1$ for all $t\in{\R^{\times}(1+\PP^l})$; that is if and only if ${\chis_i}\big|_{(\TS_l)^2}=1$. Observe that for  $0\leq i<\n$, ${\chis_i}\big|_{(\TS_l)^2}=1$, where $l\geq m$, if and only if  $\chis_i|_{(\TS\cap \KS)^2}=1$. Suppose $\chis_i|_{(\TS\cap \KS)^2}=1$, for some  $0\leq i<\n$. We show that in this case, $m=1$. Suppose $\alpha\in 1+\PP$, consider $f(X)=X^2-\alpha$. Observe that $f(1)=0\mod \PP$, and $f'(1)=2(1)\neq 0\mod p$. By Hensel's lemma, $f(X)$ has a root in $\R$; that is $\alpha\in{\R^{\times}}^2$. Therefore $1+\PP\subset {\R^{\times}}^2$, which implies $\chis_i|_{\TSL\cap \KS_1}=1$, so $m=1$. 
		\item [\textbf{The coset of $\co\lt(x\varpi^r)$}] For $\cm b$ and $\cm b'$ in $\BSL_l$,
		\(
		\cm b \:\co\lt(x\varpi^r)\cm b'=\co\lt(x\varpi^r)
		\)
		implies that $tt'\in{1+\PP^r}$ and $\zeta=\zeta'^{-1}$. Therefore, 
		\(
		\chis_i(\cm b\cm b')=\chis_i(\m b\m b',1)=\chis_i(\left(
		\begin{smallmatrix}
		tt' & ts'+st'^{-1} \\
		0 & t^{-1}t'^{-1} \\
		\end{smallmatrix}
		\right),1).
		\)
		Note that $\left(\begin{smallmatrix}
		tt' & ts'+st'^{-1} \\
		0 & t^{-1}t'^{-1} \\
		\end{smallmatrix}\right)\in \BSL\cap \KS_r $. Hence, $\chis_i(\cm b\cm b')=1$ if and only if $\BSL\cap \KS_r\subseteq \mathrm{ker}(\chis_i)$. The latter holds if and only if $r\geq m$, since $\chis_i$ is  primitive mod $m$.
	\end{description}
	Now, let us summarize our result. There is always one function with support on the identity coset, and $2(l-m)$ functions on cosets represented by
	$\co\lt(x\varpi^r), \:x\in\{1,\varepsilon\}$, $m\leq r<l$. If ${\chis_i}\big|_{(\TS\cap \KS)^2}\neq{1}$, no function in $\mathcal{H}_{i,i}$ has support on the double coset represented by $\w$, otherwise, there exists an additional function in $\mathcal H_{i,i}$ with support on the double coset represented by $\w$.
\end{proof}
Next two lemmas elaborate on the condition $\chis_i|_{(\TS\cap \KS)^2}=1$ that appears in Proposition~\ref{heckedim}. 
\begin{lem}\label{chis_icondition}
	For each $0\leq i<\n$, $\chis_i|_{(\TS\cap \KS)^2}=1$ if and only if $\chis_0|_{(\TS\cap \KS)^2}=\epsilon\circ\varpichar_{{\R^{\times}}^2}^{-2i}$.		
\end{lem}
\begin{proof}
	Let $\done{s}\in (\TS\cap \KS)^2$, so $s\in {\R^{\times}}^2$. By Lemma~\ref{distinctonAS},
	\(
	\chis_i(\done{s})=\chis_0\left(\ddd(s),\varpichar(s)^{2i}\right)=\chis_0(\done{s})\epsilon(\varpichar(s)^{2i}),
	\)
	which is equal to $1$ if and only if $\chis_0|_{(\TS\cap \KS)^2}=\epsilon\circ\varpichar_{{\R^{\times}}^2}^{-2i}$.
\end{proof}

\begin{lem}\label{varpichar^4distict}
	If $4\nmid n$ then the characters $\varpichar_{{\R^{\times}}^2}^{-2i}$, $0\leq i<\n$ are distinct. Otherwise, the $\varpichar^{-2i}_{{\R^{\times}}^2}$, $0\leq i<\frac{n}{4}$, are distinct; for $\frac{n}{4}\leq i <\frac{n}{2}$, $\varpichar^{-2i}_{{\R^{\times}}^2}=\varpichar^{-2(i-\frac{n}{4})}_{{\R^{\times}}^2}$.
\end{lem}	
\begin{proof}
	By definition of $\varpichar$ in~\eqref{varpichar}, $\varpichar^{-2i}(s)=1$ for all $s\in {\R^{\times}}^2$ if and only if $\overline{t^2}^{\frac{(q-1)2i}{n}}=1$ for all $t\in {\R^{\times}}$, or equivalently when $n|4i$.   Therefore, the equality holds only for $i=0$ unless $4|n$, in which case the equality holds for both $i=0$ and $i=\frac{n}{4}$.
\end{proof}
 For $l>m$, let $\co{W}_{i,l}$ denote the $l$-level representations 
 \(
 \co{W}_{i,l}:=(\Ind_{\BSL\cap\KSL}^{\KSL}{\chis}_i)^{\KS_l}/(\Ind_{\BSL\cap\KSL}^{\KSL}{\chis}_i)^{\KS_{l-1}}\). Moreover, for $0\leq i<\n$, set $\V i:= \Ind_{{\BSL}\cap\KSL}^{\KSL}{\chis}_{i}$. 
\begin{cor}\label{mainresultsl2}
	Assume $l\geq m$. We can decompose $\ress$ as follows:
	$$
	\ress\cong\bigoplus_{i=0}^{\n-1}\left(\V i^{\KS_{m}}\oplus\bigoplus_{l>m}\left(\co{W}^+_{i,l}\oplus\co{W}^-_{i,l}\right)\right),
	$$
where $\co{W}^+_{i,l}\oplus\co{W}^-_{i,l}\cong \co{W}_{i,l}$. All the pieces are irreducible, except when $m=1$ and $\chis_0|_{(\TS\cap \KS)^2}= \epsilon\circ\varpichar^{-2i}_{{\R^{\times}}^2}$ for some $0\leq i<\n$, in which case, we are in one of the following situations:
	\begin{enumerate}
		\item If $4\nmid n$ then there is exactly one $0\leq i<\n$ for which $\V i^{\KS_{1}}$ decomposes into two irreducible constituents. All other constituents are irreducible.
		\item If $4|n$ then there are exactly two $0\leq i,k<\n$,  $|i-k|=\frac{n}{4}$ for which $\V i^{\KS_{1}}$ decomposes into two irreducible constituents. All other constituents are irreducible.
	\end{enumerate}	
\end{cor}
\begin{proof}
	It follows from Lemma~\ref{keylammaminor} and  Proposition~\ref{heckedim} that for  $l>m$,
	\(
	\dim \Hom(\co{W}_{i,l},\co{W}_{i,l})=2.
	\)
	Hence, $\co{W}_{i,l}$ decomposes into two inequivalent irreducible subrepresentations. Moreover, 
	\begin{equation}\label{Cor-mainresultsl-local}
	\dim\Hom(\V i^{\KS_m},\V i^{\KS_m})=\begin{cases}
	1,&\text{if }\:{\chis_i}\big|_{(\TS\cap \KS)^2}\neq{1}\\
	2,&\text{otherwise.}
	\end{cases}
	\end{equation}
	By Lemma~\ref{chis_icondition}, ${\chis_i}\big|_{(\TS\cap \KS)^2}={1}$ is equivalent to $\chis_0|_{(\TS\cap \KS)^2}=\epsilon\circ\varpichar^{-2i}_{{\R^{\times}}^2}$, which also implies that $m=1$. Hence, $\V i^{\KS_m}$ is irreducible except when $m=1$ and $\chis_0|_{(\TS\cap \KS)^2}=\epsilon\circ\varpichar^{-2i}_{{\R^{\times}}^2}$, where it decomposes into two irreducible constituents. If the latter is the case, by Lemma~\ref{varpichar^4distict}, there is exactly one $0\leq i<\n$ satisfying $\chis_0|_{(\TS\cap \KS)^2}=\epsilon\circ\varpichar^{-2i}_{{\R^{\times}}^2}$ if $4\nmid n$, and there are exactly two $0\leq i<\n$ satisfying $\chis_0|_{(\TS\cap \KS)^2}=\epsilon\circ\varpichar^{-2i}_{{\R^{\times}}^2}$ if $4|n$.  
\end{proof}
Next we determine the multiplicity of each constituent in the decomposition in Corollary~\ref{mainresultsl2}. To do so, we count the dimension of 
\(
\Hom_{\KSL}\left(\Ind_{\BSL_l}^{\KSL}{\chis_k},\Ind_{\BSL_l}^{\KSL}{\chis_i}\right),
\)
which is equal to the dimension of the Hecke algebra $\mathcal{H}_{k,i}=\mathcal{H}(\BSL_l\backslash{\KSL}/\BSL_l,\chis_k,{\chis_i})$.
\begin{prop}\label{multiplicityHeckesl}	
	Let $l\geq m$, $0\leq k,i <\n$, and $i\neq k$. Then \[\dim\mathcal{H}_{k,i}=
	\begin{cases}
	2l-1,& \text{ if}\:\: {\chis_0}|_{(\TS\cap \KS)^2}=\epsilon\circ\varpichar_{{\R^\times}^2}^{-(k+i)}\\
	2(l-m),& \text{ otherwise}.
	\end{cases}
	\]

\end{prop}

\begin{proof}
	Similar to the proof of Proposition~\ref{heckedim}, we determine which double cosets in $\BSL_l\backslash{\KSL}/\BSL_l$ support a function in $\mathcal{H}_{k,i}$. For every double coset representative $\cm x$ in Lemma~\eqref{setofdoublecoset}, there exists a function $f\in\mathcal{H}_{k,i}$ with support on the double coset represented by $\cm x$ if and only if $\cm b\cm x\cm b'=\cm x$, $\cm b,\cm b'\in \BSL_l$, implies that $\chis_{k}(\cm b)\chis_{i}(\cm b')=1$. Let $t,t'\in {\mathcal{O}^{\times}}({1+\PP^l})$, $s,s'\in  \PP^l$ and $\zeta,\zeta'\in\mu_n$, so that  $
	\cm b=(\m b,\zeta)=\left(\left(\begin{smallmatrix}t&s\\
	0& t^{-1}\end{smallmatrix}\right),\zeta\right)$ and $\cm b'=(\m b',\zeta')=\left(\left(\begin{smallmatrix}t'&s'\\
	0& t'^{-1}\end{smallmatrix}\right),\zeta'\right) $ are arbitrary elements of $\BSL_l$. 
	
	Because $\chis_k\neq \chis_i$, there is no function in $\mathcal{H}_{k,i}$ with support on the identity double coset.
	
	For the double coset of $\w$, $\cm b\w\cm b'=\w$ implies that  $\m b=\m b'=\ddd(t)$, for some $t\in{\mathcal{O}^{\times}}({1+\PP^l})$ and $\zeta'=\zeta^{-1}$. Therefore, $\chis_k(\cm b){\chis_{i}}(\cm b')={\chis}_k\left(\ddd(t),\zeta\right)\chis_{i}\left(\ddd(t),\zeta^{-1}\right)$ equals
	\begin{eqnarray*}
		\chis_0\left(\ddd(t),\varpichar(t)^{2k}\zeta\right)\chis_0\left(\ddd(t),\varpichar(t)^{2i}\zeta^{-1}\right)&=&{\chis_0}\left(\ddd(t^2), \varpichar(t)^{2(k+i)}\right)
	={\chis_0}\left( \done{t^2}(\mathrm{I}_2,\varpichar(t^2)^{k+i})\right)\\
		&=&{\chis_0}\left(\done{t^2}\right)\epsilon\left(\varpichar(t^2)^{k+i}\right).
	\end{eqnarray*}
	Therefore, because $l\geq m$, $\chis_k(\cm b){\chis_{i}}(\cm b')=1$ if and only if ${\chis_0}|_{(\TS\cap \KS)^2}=\epsilon\circ\varpichar_{{\R^{\times}}^2}^{-(k+i)}$. In this case, $m=1$ and $\w$ supports a function in $\mathcal{H}_{k,i}$.

	Finally, for the double cosets represented by $\co\lt(x\varpi^r)$, $x\in \{1,\varepsilon\}$, $1\leq r<l$, $\cm b\:\co\lt(x\varpi^r)\cm b'=\co \lt(x\varpi^r)$ implies that
	$\zeta'=\zeta^{-1}$, and $t+s\varpi^r=t'^{-1}\mod \PP^l$, or equivalently, $t=t'^{-1}\mod \PP^r$, and $t^{-1}\varpi^r=\varpi^rt'^{-1}\mod \PP^l$, or equivalently $t^{-1}=t'^{-1}\mod \PP^{l-r}$. 
	Observe that, in general, $\chis_k(\cm b){\chis_{i}}(\cm b')$ is equal to 
	\begin{gather}\label{multiplicityHeckesl-local}
	{\chis}_k\left(\ddd(t),\zeta\right)\chis_{i}\left(\ddd(t'),\zeta'\right)=
	\chis_0\left(\ddd(t),\varpichar(t)^{2k}\zeta \right) \chis_0\left(\ddd(t'),\varpichar(t')^{2i}\zeta'\right)
	=\chis_0\left(\ddd(tt'),\varpichar(t)^{2k}\varpichar(t')^{2i}\zeta \zeta'\right)\\\notag
	=\chis_0\left(\done{tt'}\right)\epsilon\left(\varpichar(t)^{2k}\varpichar(t')^{2i}\zeta \zeta'\right).
	\end{gather}
	Note that $\varpichar$ is primitive mod one. Observe that $r\geq 1$ and $l-r\geq 1$. Therefore, $t=t'^{-1}\mod \PP$ and $t=t'\mod \PP$, which implies that $t=t'=\alpha\mod \PP$ where $\alpha\in \{\pm 1\}$. Hence, $\varpichar (t)^2=\varpichar(t')^2=1$, and~\eqref{multiplicityHeckesl-local} simplifies to $\chis_0\left(\done{tt'}\right)\epsilon(\zeta \zeta')$. We are in one of the following situations:
	\begin{description}
		\item [Case 1] Suppose $r\geq m$. Then we have $\zeta'=\zeta^{-1}$, and $t=t'^{-1}\mod \PP^m$; that is $tt'\in 1+\PP^m$. Hence, $\chis_0\left(\done{tt'}\right)\epsilon(\zeta \zeta')= \chis_0(tt')=1$, because $\chis_0$ is primitive mod $m$. Therefore, in this case, there is always a function in $\mathcal {H}_{k,i}$ with support on these double cosets. 
		\item [Case 2] Suppose $r<m$. Then $\zeta'=\zeta^{-1}$, so  $\chis_0\left(\done{tt'}\right)\epsilon(\zeta \zeta')= \chis_0(tt')$, which equals one if and only if $tt'\in 1+\PP^m$, which is not the case in general. Hence,  in this case, there is no function in $\mathcal {H}_{k,i}$ with support on these double cosets. 
	\end{description}
	To summarize the result, the coset represented by $\w$ supports a function in $\mathcal{H}_{k,i}$ if and only if $\chis_0|_{(\TS\cap \KS)^2}=\epsilon\circ\varpichar_{{\R^\times}^2}^{-(k+i)}$. If $r\geq m$  then the cosets represented by $\lt(x\varpi^r)$ support a function in $\mathcal{H}_{k,i}$; otherwise, there is no function in $\mathcal {H}_{k,i}$ with support on these double cosets. 
\end{proof}
\begin{cor}\label{multiplicitysl}
	In the decomposition of $\Res_{\KSL}\Ind_{\BSL}^{\co \SL}\pis$ given in Corollary~\ref{mainresultsl2}, 
	\begin{enumerate}
		\item For  each $0\leq i<\n$ and $l>m$, there exists a way of decomposing $\co W_{i,l}$ as $\co W_{i,l}^+\oplus\co W_{i,l}^-$ such that for $l>m$, $\co W_{i,l}^{+}\cong\co W_{j,l}^{+}$ and $\co W_{i,l}^{-}\cong \co W_{j,l}^{-}$ for all $ 0\leq i,j<\n$.
		\item For $l=m$, $\{(\Ind_{{\BSL}\cap\KSL}^{\KSL}{\chis}_{i})^{\KS_{m}}\mid 0\leq i<\n\}$ consists of mutually inequivalent representations, except when $m=1$ and ${\chis_0}|_{(\TS\cap \KS)^2}=\epsilon\circ\varpichar_{{\R^{\times}}^2}^{-j}$, for some $0\leq j<\n$, where \(\V i^{\KS_{1}}\cong \V k^{\KS_{1}},\) exactly when $i+k\equiv j\mod \n$.
	\end{enumerate} 
\end{cor}
\begin{proof}
	It follows from Proposition~\ref{multiplicityHeckesl} that for $l>m$,  
	\(
	\dim \Hom_{\KSL}\left(\co W_{i,l},\co W_{k,l}\right)=2,
	\)
	and when $i+k\equiv j\mod \n$
	\[
	\dim \Hom _{\KSL}\left(\V i^{\KS_{m}},\V k^{\KS_{m}}\right)=\begin{cases}
	1,& {\chis_0}|_{(\TS\cap \KS)^2}=\epsilon\circ\varpichar_{{\R^{\times}}^2}^{-j}\\
	0, &\text{otherwise,}
	\end{cases}
	\]
	and hence the result. 	
\end{proof}

In order to further investigate the irreducible spaces $\co{W}^+_{i,l}$ and $\co{W}^-_{i,l}$, we will show that $\co{W}_{i,l}$, $0\leq i< \n$, is the restriction to $\KSL$ of an irreducible representation of the maximal compact subgroup $\co\KG$ of the covering group $\co \GL$ of $\mathrm{GL}_2(\F)$.


\section{Branching Rules for $\co\GL$}
\label{branching-rule-sectionGL}
 We define the genuine principal series representations of $\co \GL$ similarly by starting with a genuine smooth irreducible representation $\pig$ of $\TGL$ with the central character $\chig$, which is constructed via the Stone-von Neumann theorem. Observe that $\dim{\pig}=[\TGL: \AG]=n^2$. Then, after extending $\pig$ trivially over $\UNG$, the genuine principal series representation $\psrg$ of $\co \GL$ is $\Ind_{\BGL}^{\co \GL}\pig$. Applying a similar machinery as in Section~\ref{branching-rule-section}, we obtain the $\mathsf{K}$-type decomposition for $\Res_{\KGL}\psrg$. Since the argument in Section~\ref{branching-rule-section} goes through almost exactly, here we only overview the main steps and  point out the differences. For detailed calculations, see~\cite{Camelia-PHDthesis}.
 
 Similar to Lemma~\ref{distinctonAS}, it follows that $\Res_{\AG}\pig\cong \bigoplus_{i,j=0}^{n-1}\chig_{i,j}$,
 where the $\chig_{i,j}$ denote $n^2$ distinct characters of $\AG$, defined by
 \(
 \chig_{i,j}\left(\ddd(a\varpi^{un},b\varpi^{vn}),\zeta\right)=\chig_0\left(\ddd(a\varpi^{un},b\varpi^{vn}),\varpichar(a)^{-j}\varpichar(b)^{-i}\zeta\right)
 \) 
 where $a,b\in \R^{\times}$, $u,v\in\mathbb{Z}$ and $\zeta\in\mu_n$ and $\varpichar(a)=(\varpi,a)_n$ was defined in~\eqref{varpichar}, and $\chig_0$ is a fixed extension of $\chig$ to $\AG$. The $\chig_{i,j}$ remain distinct when restricted to $\TGL\cap \KGL$, and again writing $\chig_{i,j}$ for there restrictions, $\Res_{\TGL\cap\co{\KG}}\pig\cong\bigoplus_{i,j=0}^{n-1}\chig_{i,j}$. Then similar to Proposition~\ref{ress}, we have $
 \resg\cong\bigoplus_{i,j=0}^{n-1}{\Ind_{{\BGL}\cap{\KGL}}^{{\KGL}}{\chig_{i,j}}}.$ This latter isomorphism reduces the problem of decomposing the $\mathsf{K}$-type to the one of decomposing each $\Ind_{\co{\BG}\cap\co{\KG}}^{\co{\KG}}{\chig_{i,j}}$, which, by smoothness, can be written as the union of its $\KG_l$, $l\geq 1$, fixed points. 
 
 Suppose $\chig$ is primitive mod $m$. It follows that the $\chig_{i,j}$ are also primitive mod $m$. Set $\BGL_l=(\BGL\cap \KGL)\KG_l$. It can be seen that each level $l$ representation $\left(\Ind_{\BGL\cap\KGL}^{\KGL}{\chig_{i,j}}\right)^{\KG_l}=\Ind_{\BGL_l}^{\KGL}{\chig_{i,j}}$ if $l\geq{m}$, and is zero if $l<m$. Similar to Proposition~\ref{hecke}, one can see that  \(\dim\Hom_{\KGL}(\Ind_{\BGL_l}^{\KGL}\chig_{i,j},\Ind_{\BGL_l}^{\KGL}\chig_{i,j})=\dim \mathcal{H'}_{i,j}({\BGL_l}\backslash{\KGL}/\BGL_l,\chig_{i,j},\chig_{i,j})\).
We count the dimension of $\mathcal{H'}_{i,j}$ using a method similar to the one we used in Proposition~\ref{heckedim}. To do so, we need to calculate a set of double coset representatives of $\BGL_l$ in $\KGL$.
 \begin{lem}\label{doublecosetrepGL}
 A complete set of double coset representatives of $\BGL_l$ in $\KGL$ is given by $
 \{(\mathrm{I}_2,1),\co w, \co \lt(\varpi^r)\mid 1\leq{r}<l\}.$
 \end{lem}
 \begin{proof}
 Note that this set is a subset of the set $S$ in~\eqref{setofdoublecoset}. Observe that under the isomorphism 
 \begin{equation}\label{doublecosetrepGL-local}
 \F^{\times}\ltimes\co \SL\cong\co \GL, \quad 
 \left(y,(\m g,\zeta)\right)\mapsto \left(\ddd(1,y)\m g,\zeta\right),
 \end{equation}
 $\R^{\times}\times\KSL$ maps to $\KGL$ and $\R^{\times}\times {\BSL_{l}}$ maps to $\BGL_{l}$. For every $\cm k'\in \KGL$, let $(y,\cm k)$ be the inverse image of $\cm k'$ under the isomorphism~\eqref{doublecosetrepGL-local}, and let $\cm b_1,\cm b_2\in \BSL_{l}$ be such that $\cm b_1\cm x\cm b_2=\cm k$, for some $\cm x\in S$. Let $\cm b'_1$ and $\cm b'_2$ be the image of $(y,\cm b_1)$ and $(y,\cm b_2)$ under~\eqref{doublecosetrepGL-local} respectively. It follows from the multiplication of $ \F^{\times}\ltimes\co \SL$ and the isomorphism map~\eqref{doublecosetrepGL-local}, that $\cm b'_1\cm x\cm b'_2=\cm k'$. Thus, $\KGL=\bigcup_{\cm x\in S}\BGL_{l}\cm x\BGL_{l}$. A short calculation shows that
 $$
 \left(\ddd(\varepsilon^{-1},1),1\right)
 \co\lt(\varpi^r)\left(\ddd(\varepsilon,1),1\right)=\left(\lt(\varepsilon\varpi^r),(\varpi^r,\varepsilon)_n(\varepsilon,\varpi^r)_n\right)=\co\lt(\varepsilon\varpi^r),
 $$
 where $\varepsilon$ is a fixed non-square and $1\leq{r}<l$. It is not difficult to see that other cosets of $S$ remain distinct in $\KGL$.
 \end{proof}
The following proposition can be proved similar to Proposition~\ref{heckedim}.
\begin{prop}\label{heckedimGL}
	Let $l\geq m$. Then 
	\(
	\mathrm{dim}\mathcal{H'}_{i,j}=
	\begin{cases}
	1+(l-m),&\text{if}\quad {\chig_{i,j}\big|_{\TG\cap \KG}\neq{1}};\\
	2+(l-m),&\text{otherwise.}
	\end{cases}
	\)
\end{prop}	
\begin{lem}
	For $0\leq i,j<n$, $\chig_{i,j}|_{\TG\cap \KG}=1$  if and only if $\chig_{0,0}|_{\TG\cap \KG}=\epsilon\circ \varpichar^{j-i}$.
\end{lem}
\begin{proof}
Note that
	\(\chig_{i,j}\left(\done{a}\right)=\chig_{0,0}\left(\ddd(a), \varpichar^{i-j}(a)\right)\), which is equal to $1$ if and only if $\chig_{0,0}|_{T\cap K}=\epsilon\circ\varpichar^{j-i}$.
\end{proof}
For $l>m$, let  $\co{W'}_{i,j,l}$ denote the $l$-level quotient representation $(\Ind_{\co{\BG}\cap\co{\KG}}^{\co{\KG}}{\chig}_{i,j})^{\KG_l}/(\Ind_{\co{\BG}\cap\co{\KG}}^{\co{\KG}}{\chig}_{i,j})^{\KG_{l-1}}$. The $\mathsf{K}$-type decomposition $\resg$ is given in the following Corollary. 
\begin{cor}\label{mainresultgl2}
	We can decompose $\resg$ as follows:
	\begin{equation}\label{mainresultgl2-equ}
	\resg\simeq\bigoplus_{i,j=0}^{n-1}\left((\Ind_{{\BGL}\cap\co{\KG}}^{\co{\KG}}{\chig}_{i,j})^{\KG_{m}}\oplus\bigoplus_{l>m}\co{W'}_{{i,j},l}\right).
	\end{equation}
	If $\chig_{0,0}|_{\TG\cap \KG}\neq\varpichar^{k}|_{\R^\times}$, for all $0\leq k<n$, then all the pieces are irreducible. Otherwise, there are exactly $n$ pairs $(i,j)$, $0\leq i,j<n$, such that $j-i\equiv k\mod n$, and $(\Ind_{{\BGL}\cap\co{\KG}}^{\co{\KG}}{\chig}_{i,j})^{\KG_{m}}$ decomposes into two irreducible constituents. The rest of the constituents are irreducible. 
\end{cor}
\begin{proof}
It follows from Proposition~\ref{heckedimGL} that for $0\leq i,j<n$, $(\Ind_{{\BGL}\cap\co{\KG}}^{\co{\KG}}{\chig}_{i,j})^{\KG_{m}}$ is irreducible if $\chig_{0,0}|_{\TG\cap \KG}\neq \varpichar|_{\R^\times}^{j-i}$, and decomposes into two inequivalent constituents otherwise. Moreover, for $l>m$, the quotients $\co{W'}_{i,j,l}$  are irreducible. Note that the map $(i,j)\to j-i\mod n$ has a kernel of size $n$. Hence, if there exists a pair such that $\chig_{0,0}|_{\TG\cap \KG}=\varpichar|_{\R^\times}^{j-i}$, then there are exactly $n$ distinct such pairs.
\end{proof}
\subsection{Restriction of $\Ind_{\BGL}^{\co\GL}\pig$ to $\KSL$}
Fix a genuine irreducible representation $\pis$ of $\TSL$ with central character $\chis$, where $\chis$ is primitive mod $m$. Let $\co{W}_{k,l}$, $\co{W}^+_{k,l}$, and $\co{W}^-_{k,l}$ be the representations of $\KSL$ that appear in the $\mathsf{K}$-type decomposition of $\ress$ in Corollary~\ref{mainresultsl2}. In this section, we show that, for each $0\leq k<\n$, $\co{W}_{k,l}\cong \Res_{\KSL}W'$, where $W'$ is some irreducible representation of $\co \KG$. We deduce that $\co{W}^+_{k,l}$ and $\co{W}^-_{k,l}$ have the same dimension. 

Let $\pig$ be a genuine irreducible representation of $\TGL$ with central character $\chig$, such that depth of $\chig$ is equal to depth of $\chis$, and that $\pis$ appears in $\Res_{\TSL}\pig$. Let $\chig_{i,j}$, $0\leq i,j<n$ be all possible extensions of $\chig$ to $\AG$. To find $W'$, we consider the restriction of the principal series representation $\Ind_{\BGL}^{\co\GL}\pig$ to $\KSL$. Because the structure of the $\TSL$ depends on the parity of $n$, we consider the cases for even and odd $n$ separately.
\subsubsection{$n$ odd}
Recall that for odd $n$ we have	\(Z(\TSL)=\{\left(\ddd(a),\zeta\right)\mid a\in{\F^{\times}}^n,\zeta\in\mu_n\},\: \AS=\{\left(\ddd(a),\zeta\right)\mid a\in{\F}^\times,\zeta\in{\mu_n},\:n|\val(a)\},\:
	\TSL\cap\KSL= \{\left(\ddd(a),\zeta\right)\mid a\in \R^{\times},\zeta\in{\mu_n}\},\:
	Z(\TGL)= \{\left(\ddd(a,b),\zeta\right)\mid a,b\in{\F^{\times}}^n,\zeta\in\mu_n\},\:
	\AG=\{\left(\ddd(a,b),\zeta\right)\mid a,b\in{\F^\times},\zeta\in{\mu_n},n|\val(a),n|\val(b)\}, \:
	\TGL\cap \co \KG=\{\left(\ddd(a,b),\zeta\right)\mid a,b\in{\R^{\times}},\zeta\in\mu_n\}.\) Observe that $Z(\TGL)\cap \TSL=Z(\TSL)$ and $\AG\cap \TSL=\AS$.

We compute $\Res_{\KSL}\Res_{\co\KG}\Ind_{\BGL}^{\co\GL}\pig$, where the decomposition of $\Res_{\co\KG}\Ind_{\BGL}^{\co\GL}\pig$ is given in Corollary~\ref{mainresultgl2}. The assumption $\pis$ appears in $\Res_{\TSL}\pig$ implies $\chig|_{Z(\TSL)}=\chis$. We further assume that the choice of $\chis_0$ is such that $\displaystyle {\Res_A\chig_0=\chis_0}$. In order to study the restriction of each piece in~\eqref{mainresultgl2-equ}, we need to restrict the characters $\chig_{i,j}$ to $\TSL\cap \KSL$. 

\begin{lem}\label{chigtochis}
	Assume $n$ is odd. For $0\leq i,j< n$, let $k$ be the integer in $\{0,\cdots,n-1\}$ such that $k\equiv\frac{i-j}{2}\mod n$. Then $\Res_{\TSL\cap\KSL}\chig_{i,j}=\chis_{k}$.
\end{lem}
\begin{proof}
	Let $\left(\ddd(u),\zeta\right)\in\TSL\cap \KSL$. Then \(\chig_{i,j}\left(\ddd(u),\zeta\right)=\chig_0\left(\ddd(u),\varpichar(u)^{i-j}\zeta\right),\)
	which by Lemma~\ref{distinctonAS}, and because $\chig_0|_{\AS}=\chis_0$, is equal to $\chis_0\left(\ddd(u),\varpichar(u)^{2k}\zeta\right)=\chis_{k}\left(\ddd(u),\zeta\right)$.
\end{proof}
The cardinality of the kernel of the map $(i,j)\to k\mod n$, in Lemma~\ref{chigtochis}, is $n$; that is for each $k$, there are exactly $n$ distinct characters $\chig_{i,j}$ of $\TGL\cap \co \KG$ that restrict to $\chis_k$ on $\TSL\cap \KSL$. 
\begin{lem}\label{secondapproachlemodd}
	Assume $n$ is odd. Let $i,j$ and $k$ be in $\{0,\cdots,n-1\}$, such that $\chig_{i,j}|_{\TSL\cap \KSL}=\chis_k$. Then, for all $l\geq m$,
	\(
	\Res_{\KSL}\left(\Ind_{\co\BG\cap\co\KG}^{\co\KG}\chig_{i,j}\right)^{\KG_l}\cong\left(\Ind_{\BSL\cap{\KSL}}^{\KSL}\chis_k\right)^{\KS_{l}}.
	\)
\end{lem}
\begin{proof}
	It is enough to show that $\Res_{\KSL}\Ind_{\co{\BG}_l}^{\co{\KG}}\chig_{i,j}\cong \Ind_{\BSL_l}^{\KSL}\chis_k$. 
	Note that $\KSL\backslash {\KGL}/\co{\BG}_l$ is trivial and $\co\BG_l\cap\KSL=\BSL_l$. So by Mackey's theory, we have
	\(
	\Res_{\KSL}\Ind_{{\BGL}_l}^{\co{\KG}}\chig_{i,j}\cong \Ind_{\BSL_l}^{\KSL}\Res_{\BSL_l}\chig_{i,j},
	\)
	which is equal to $\Ind_{\BSL_l}^{\KSL}\chis_k$ by choice of $i,j$ and $k$. 
\end{proof}

\subsubsection{$n$ even}
Recall that for even $n$, 
	\(Z(\TSL)=\{(\ddd(t),\zeta)\mid t\in{\F^{\times}}^{n/2},\zeta\in\mu_n\}\), \(\AS=\{(\ddd(t),\zeta)\mid t\in{\F^{\times}},\textstyle{\frac{n}{2}}|\val(t),\zeta\in\mu_n\}\), 
	\(\TSL\cap\KSL=\{(\ddd(t),\zeta)\mid t\in{\R^{\times}},\zeta\in\mu_n\}\),  \(Z(\TGL)=\{(\ddd(t,s),\zeta)\mid t,s\in{\F^{\times}}^n,\zeta\in\mu_n\}\), \(\AG=\{(\ddd(t,s),\zeta)\mid t,s\in{\F^{\times}},n|\val(t),n|\val(s),\zeta\in\mu_n\}\), 
	\(\TGL\cap \KGL=\{(\ddd(t,s),\zeta)\mid t,s\in{\R^{\times}},\zeta\in\mu_n\},
\)
and therefore,
\(
	Z(\TGL)\cap\TSL=\{(\ddd(t),\zeta)\mid t\in {\F^{\times}}^n,\zeta\in\mu_n\}\), and \(
	\AG\cap\TSL=\{(\ddd(t),\zeta)\mid t\in{\F^{\times}}, n|\val(t),\zeta\in\mu_n\}.\)

Unlike the case for odd $n$, the centre $Z(\TGL)$ and the maximal abelian subgroup $\AG$ of $\TGL$ do not restrict to those of $\TSL$ upon restriction to $\TSL$.  Observe that $[Z(\TSL):Z(\TGL)\cap\TSL]=4$, $[\AS:\AG\cap\TSL]=2$. This mismatch makes the computation of $\Res_{\KSL}\Ind_{\BGL}^{\co\GL}\pig$ more delicate.  Indeed, our assumption that $\pis$ appears in $\Res_{\TSL}\pig$ does not imply that $\pig$ is $\pis$ isotypic, upon restriction to $\TSL$. We show that $\pis$ is one of the four distinct irreducible representations of $\TSL$ that appear in $\Res_{\TSL}\pig$. 

Set $\underline{\chi}:=\Res_{Z(\TGL)\cap\TSL}\chig$. Note that $|\n\mathbb{Z}/n\mathbb{Z}|=|{\R^\times}^{\n}/{\R^{\times}}^n|=2$. We denote the coset representatives of the former by $\{e,o\}$.  Let $L$ denote the set of coset representatives for $Z(\TSL)/(Z(\TGL)\cap\TSL)$, so $|L|=4$. The representation $\Ind_{Z(\TGL)\cap\TSL}^{Z(\TSL)}\underline{\chi}$ decomposes into $4$ distinct characters ${}_{\ell}\chis$:
\begin{equation}\label{chisi}
\Ind_{Z(\TGL)\cap\TSL}^{Z(\TSL)}\underline{\chi}=\bigoplus_{\ell\in L}{}_{\ell}\chis.
\end{equation}

We denote the irreducible genuine representation of $\TSL$ with central character ${}_{\ell}\chis$ by $\pis_{\ell}$.

\begin{prop}\label{restoTeven}
	Assume $n$ is even. Let ${}_\ell\chis$, $\ell\in L$ be as in~\eqref{chisi}. Then
	\(
	\Res_{\TSL}\pig=\bigoplus_{\ell\in L}\left[\left(\pis_\ell\right)^{\oplus n/2}\right],
	\)
	where $\pis_{\ell}$ are mutually inequivalent and $\pis\cong\pis_{\ell}$ for some $\ell\in L$. 
\end{prop}

\begin{proof}
	Note that $X=\{\left(\ddd(1, \varpi^j) ,1\right)\:|\:0\leq j<{n}\}$ is a system of coset representatives for $\TSL\backslash\TGL/\AG$, and that $\AG$ is stable under conjugation by $\cm x\in X$. Moreover, it is not difficult to see that for $\cm x=\left(\ddd(1, \varpi^j ),1\right)$, ${\chig_0}^{\cm x}={\chig_{0,j}}$. Therefore, by Mackey's theory,
	\[
	\Res_{\TSL}\pig=\bigoplus_{\cm x\in X}\left(\Ind_{(\TSL\cap {\AG}^{\cm x})}^{\TSL}{\chig_{0}}^{\cm x}\right)\\\notag
	=\bigoplus_{j=0}^{n-1}\Ind_{\AS}^{\TSL}\left(\Ind_{\TSL\cap\AG}^{\AS}{\chig_{0,j}}\right).
	\]
	Observe that $[\AS:\TSL\cap\AG]=2$, with coset representatives $\{e,o\}$. Therefore, for every $0\leq j<{n}$, $\Ind_{\TSL\cap\AG}^{\AS}\chig_{0,j}$ is a $2$-dimensional representation of the abelian group $\AS$ and hence decomposes into direct sum of two characters: ${}_e\chig_j\oplus{}_o\chig_{j}$.
	
	Next, we show that the elements of the set $\{{}_e\chig_j,{}_o\chig_{j}\mid 0\leq j< n\}$ are distinct. Note that for  $0\leq j< {n}$, $\Res_{\TSL\cap \AG}\Ind_{\TSL\cap\AG}^{\AS}{\chig_{0,j}}\cong \chig_{0,j}\oplus \chig_{0,j}$. Suppose $0\leq i,j< n$, by Frobenius reciprocity
	\[
		\Hom_{\AS}\left(\Ind_{\TSL\cap\AG}^{\AS}{\chig_{0,j}},\Ind_{\TSL\cap\AG}^{\AS}{\chig_{0,i}}\right)=
		\Hom_{\TSL\cap \AG}\left(\Res_{\TSL\cap \AG}\Ind_{\TSL\cap\AG}^{\AS}{\chig_{0,j}},\chig_{0,i}\right)
		=\Hom_{\TSL\cap \AG}\left(\chig_{0,j}\oplus \chig_{0,j},\chig_{0,i}\right).
	\]
	We can easily see that $\chig_{0,j}$ and $\chig_{0,i}$ coincide on $\TSL\cap \AG$ if and only if $i=j$. Whence, 
	\[
	\dim\Hom_{\AS}\left(\Ind_{\TSL\cap\AG}^{\AS}{\chig_{0,j}},\Ind_{\TSL\cap\AG}^{\AS}{\chig_{0,i}}\right)=\left\{
	\begin{array}{ll}
	2, & i=j \\
	0, & \text{otherwise.}
	\end{array}
	\right.
	\]
	Therefore,  the elements of $\{{}_e\chig_{j},{}_o\chig_{j}\mid 0\leq j< n\}$ are $2n$ distinct characters of $A$, which  because $[A:Z(\TSL)]=n/2$, implies that they restrict to, at least $4$, distinct characters upon restriction to $Z(\TSL)$. Moreover, because $\pis$ appears in $\Res_{\TSL}\pig$, at least one of these $4$ central characters is $\chis$. Observe that, for  $0\leq j< n$, and $\alpha\in{\{e,o\}}$, \(\Res_{Z(\TGL)\cap\TSL}\:\:{}_\alpha\chig_{j}=\underline\chis\).
	
	Consider 
	\(
	\Ind_{Z(\TGL)\cap\TSL}^{\AS}\underline \chis=\Ind_{Z(\TSL)}^{\AS}\Ind_{Z(\TGL)\cap\TSL}^{Z(\TSL)}\underline \chis=\Ind_{Z(\TSL)}^{\AS}\bigoplus_{\ell\in L}{}_\ell\chis=\bigoplus_{\ell\in L, 0\leq k<n/2}{}_\ell\chis_k.
	\)
	Observe that, the ${}_\ell\chis_k$, are $2n$ distinct characters that restrict to $\underline \chis$  on $Z(\TGL)\cap\TSL$, and exhaust every such character. Hence, the sets $\{{}_e\chig_{0,j},{}_o\chig_{0,j}\mid 0\leq j< n\}$ and $\{{}_\ell\chis_k\mid \ell\in L,0\leq k<n/2 \}$ are equal. In particular,
	\(
		\Res_{\TSL}\pig\cong\Ind_{\AS}^{\TSL}\bigoplus_{0\leq j< {n}}{}_e\chig_j\oplus{}_o\chig_{j}
		= \Ind_{\AS}^{\TSL}\left(\bigoplus_{\ell\in L, 0\leq k<n/2}{}_\ell\chis_k\right)
		\cong\bigoplus_{\ell\in L}\pis_\ell^{\oplus \frac{n}{2}}.	
	\)
	The last equality is because the ${}_\ell\chis_k$ extend ${}_\ell\chis$. Moreover, $\pis_{\ell}$ are mutually inequivalent because ${}_{\ell}\chis$ are mutually inequivalent. Finally, because $\chis={}_{\ell}\chis$ for some $\ell\in L$,  $\pis\cong\pis_{\ell}$ for some $\ell\in L$. 
\end{proof}

We compute $\Res_{\KSL}\Res_{\co \KG}\Ind_{\BGL}^{\co \GL}\pig$. First, we need to study $\Res_{\TSL\cap\KSL}\chig_{i,j}$. 
\begin{lem}\label{Reschigijformula}
	Let $\chig_{i,j}$, $0\leq i,j<n$ be as in~\eqref{mainresultgl2-equ}. Then 
	\(
	\Res_{\TSL\cap \KSL}\chig_{i,j}\left(\ddd(t),\zeta\right)=\chig_{0,0}\left(\ddd(t),\varpichar(t)^{i-j}\zeta\right),
	\)	
	for all $\left(\ddd(t), \zeta\right)\in\TSL\cap\KSL$.  
\end{lem}	
\begin{proof}
	Let $\left(\ddd(t) \zeta\right)\in\TSL\cap\KSL$. Then
	\[
	\Res_{\TSL\cap \KSL}\chig_{i,j}\left(\ddd(t),\zeta\right)=\chig_{0,0}\left(\ddd(t),\varpichar(t)^{-j}\varpichar(t)^i\zeta\right)=\chig_{0,0}\left(\ddd(t),\varpichar(t)^{i-j}\zeta\right).
	\]
\end{proof}
Therefore, $\{\Res_{\TSL\cap \KSL}\chig_{i,j}\mid 0\leq i,j<n\}$ consists of $n$ distinct characters of  ${\TSL\cap \KSL}$. In the next lemma and proposition, we realize these characters as characters of $\TSL\cap \KSL$ that come from central characters ${}_\ell \chis$, $\ell\in L$, of $Z(\TSL)$.
\begin{lem}\label{chigjinTSL}
	Each $\Res_{\TSL\cap \KSL}\chig_{i,j}$ appears exactly twice in $\bigoplus_{\ell\in L, 0\leq k<\frac{n}{2}}\:{}_\ell\chis_k$.  	
\end{lem}	
\begin{proof}
	Note that 
	\(
	\Res_{\TSL\cap\KSL}\Ind_{Z(\TGL)\cap \TSL}^{\AS}\underline{\chis}=\Res_{\TSL\cap\KSL}\left(\bigoplus_{\ell\in L,0\leq k<\frac{n}{2} }{}_\ell\chis_k\right).
	\) 
	Consider 
	\begin{equation}\label{chigjinTSL-local}
	\Hom_{\TSL\cap \KSL}\left(\Res_{\TSL\cap \KSL}\chig_{i,j}, \Res_{\TSL\cap\KSL}\Ind_{Z(\TGL)\cap \TSL}^{\AS}\underline{\chis}\right).
	\end{equation}
	Observe that $\TSL\cap \KSL\backslash \AS/Z(\TGL)\cap\TSL\cong \n \mathbb{Z}/n\mathbb Z$. So, by Mackey's theory and Frobenius reciprocity~\eqref{chigjinTSL-local} is 
	\begin{gather*}
	{\Hom_{\TSL\cap \KSL}\left(\Res_{\TSL\cap \KSL}\chig_{i,j},\left(\Ind_{Z(\TGL)\cap \KSL}^{\TSL\cap \KSL}\underline \chis\right)^{\oplus 2}\right)
		\cong\Hom_{{Z(\TGL)\cap \KSL}}\left(\Res_{Z(\TGL)\cap \KSL}\chig_{i,j},\Res_{Z(\TGL)\cap \KSL}\underline \chis^{\oplus 2}\right).}
	\end{gather*}
	Because $\Res_{Z(\TSL)}\chig=\underline\chis$, for all $0\leq i,j< n$, $\Res_{Z(\TGL)\cap \KSL}\chig_{i,j}=\Res_{Z(\TGL)\cap \KSL}\underline \chis$, and hence,~\eqref{chigjinTSL-local} is $2$-dimensional, which shows that $\Res_{\TSL\cap \KSL}\chig_{i,j}$ appears exactly twice in $\bigoplus_{\ell\in L, 0\leq k<\n}\:{}_\ell\chis_k$.
\end{proof}
Note that the map $(i,j)\to i-j \mod n$, which appears in Lemma~\ref{Reschigijformula}, has a kernel of size $n$. Therefore, it is easy to see that $\{\Res_{\TSL\cap \KSL}\chig_{0,j}\mid 0\leq j<n\}$ consists of $n$ distinct characters of $\TSL\cap \KSL$, each appearing exactly twice in $\bigoplus_{\ell\in L, 0\leq k<\n}\:{}_\ell\chis_k$ by Lemma~\ref{chigjinTSL}. By a simple counting argument, we deduce that for every $0\leq k<\n$ and $\ell\in L$, there exists a $0\leq j<n$, such that $\Res_{\TSL\cap \KSL}\chig_{0,j}={}_\ell\chis_k$.
Similar to Lemma~\ref{secondapproachlemodd}, we see that, for $n$ even, if
 $0\leq j<n$, $0\leq k<\n$ and $\ell\in L$ are such that  $\Res_{\TSL\cap \KSL}\chig_{0,j}={}_\ell\chis_k$, then, for all $l\geq m$,
\(
	\left(\Ind_{\BSL\cap{\KSL}}^{\KSL}{}_{\ell}\chis_k\right)^{\KS_{l}}\cong\Res_{\KSL}\left(\Ind_{\co\BG\cap\co\KG}^{\co\KG}\chig_{0,j}\right)^{\KG_l}.
\)

The following proposition sums up the result in this section. 

\begin{prop}\label{secondapproachpropneven}
	 Let $\pis$ and $\pig$ be irreducible representations of $\TSL$ and $\TGL$ with central characters $\chis$ and $\chig$, primitive mod $m$,  respectively, such that $\pis$ appears in $\Res_{\TSL}\pig$. For $l>m$, $0\leq k<\n$, $0\leq i,j<n$, let $\co W_{k,l}=\co W_{k,l}^{-}\oplus\co W_{k,l}^{+}$ and  $\co{W'}_{{i,j},l}$ be the quotient spaces that appear in the decompositions in Corollary~\ref{mainresultsl2} and Corollary~\ref{mainresultgl2} respectively. Then, for each $0\leq k<\n$, $l>m$, $\co W_{k,l}=\Res_{\KSL}\co {W'}_{i,j,l}$, for some $0\leq i,j<n$. 
	
\end{prop}
\begin{proof}
If $n$ is odd, it follows from Lemma~\ref{secondapproachlemodd} that for a given $k$ and $l$ there exists $0\leq i,j<n$ such that \(
\left(\Ind_{\BSL\cap{\KSL}}^{\KSL}\chis_k\right)^{\KS_{l}}\cong \Res_{\KSL}\left(\Ind_{\co\BG\cap\co\KG}^{\co\KG}\chig_{i,j}\right)^{\KG_l}\). Without loss of generality, we can assume $i=0$. For $n$ even, it follows from Proposition~\ref{restoTeven} that $\chis={}_\ell\chis$ for some $\ell\in L$, where ${}_\ell\chis$ are defined in~\eqref{chisi}. It is a consequence of Lemma~\ref{chigjinTSL} that, for a given $k$ and $l$ there exists $0\leq j<n$ such that 	$\left(\Ind_{\BSL\cap{\KSL}}^{\KSL}{}_{\ell}\chis_k\right)^{\KS_{l}}\cong\Res_{\KSL}\left(\Ind_{\co\BG\cap\co\KG}^{\co\KG}\chig_{0,j}\right)^{\KG_l}$. Consider $\co W'_{{0,j},l} =(\Ind_{\co{\BG}\cap\co{\KG}}^{\co{\KG}}{\chig}_{0,j})^{\KG_l}/(\Ind_{\co{\BG}\cap\co{\KG}}^{\co{\KG}}{\chig}_{0,j})^{\KG_{l-1}}$. Observe that 

\begin{eqnarray*}
\Res_{\KSL}\co W'_{{0,j},l}&=&\Res_{\KSL}\left[(\Ind_{\co{\BG}\cap\co{\KG}}^{\co{\KG}}{\chig}_{0,j})^{\KG_l}\right]/\Res_{\KSL}\left[(\Ind_{\co{\BG}\cap\co{\KG}}^{\co{\KG}}{\chig}_{0,j})^{\KG_{l-1}}\right]\\
&=&\left(\Ind_{\BSL\cap{\KSL}}^{\KSL}\chis_k\right)^{K_{l}}/\left(\Ind_{\BSL\cap{\KSL}}^{\KSL}\chis_k\right)^{K_{l-1}}=\co W_{k,l}^{-}\oplus\co W_{k,l}^{+}.
\end{eqnarray*}	 
\end{proof}
\begin{cor}\label{Wpmsamedimension}
	The inequivalent irreducible representations $\co W_{k,l}^{-}$ and $\co W_{k,l}^{+}$, $0\leq k<\n$, $l>m$,  that appear in the $\mathsf{K}$-type decomposition $\Res_{\KSL}\Ind_{\BSL}^{\co \SL} \pis$ in Corollary~\ref{mainresultsl2} are of the same dimension. 
\end{cor}
\begin{proof}
	By Proposition~\ref{secondapproachpropneven}, for any  $0\leq k<\n$, $l>m$, $\co W_{k,l}=\co W_{k,l}^{-}\oplus\co W_{k,l}^{+}$, is restriction of some irreducible representation $\co W'_{i,j}$ of $\co \KG$, for some $0\leq i,j<n$. Hence, there exists an element of $\co\KG\setminus \KSL$ that maps $\co W_{k,l}^{-}$ to $\co W_{k,l}^{+}$ bijectively. 
\end{proof}
\section {Main Result}\label{mainresultsection}
Finally, we put all of our results together to make the main result  of this paper.

\begin{theo}\label{MainTheoremofK-types}
	Let $\pis$ be a genuine irreducible representation of $\TSL$ with central character $\chis$, primitive mod $m$, and let $\chis_k$, $0\leq k<\n$, be all the possible extensions of $\chis$ to $\AS$.  Then 		
	\[
	\ress\cong\bigoplus_{k=0}^{\n-1}\left(\V k^{\KS_{m}}\right)\oplus\bigoplus_{l>m}\left(\co{W}^+_{0,l}\oplus\co{W}^-_{0,l}\right)^{\oplus \n},
	\]
	where $\co{W}^+_{0,l}$ and $\co{W}^-_{0,l}$ are two inequivalent irreducible representations of $\KSL$ with the same dimension, and $\left(\co{W}^+_{0,l}\oplus\co{W}^-_{0,l}\right)\cong(\Ind_{{\BSL}\cap\KSL}^{\KSL}{\chis}_{0})^{\KS_{l}}/(\Ind_{{\BSL}\cap\KSL}^{\KSL}{\chis}_{0})^{\KS_{l-1}}$, and $\V k=\Ind_{\BSL\cap \KSL}^{\KSL}\chis_k$. 
	
	We consider $\TS\cap \KS$ as a subgroup of $\TSL$. The $m$-level representations
	\(\V k^{\KS_{m}}\),
	where $0\leq k<\n$, are irreducible and mutually inequivalent, except when $m=1$, and for some $0\leq k< \n$, $\chis_k|_{(\TS\cap \KS)}$ is a quadratic character. In this case, up to relabelling, we can assume that $\chis_0|_{(\TS\cap \KS)}$ is a quadratic character, and we are in one of the following situations: 
\begin{enumerate}
	\item If $4\nmid n$ then  \(\V k^{\KS_{1}}\) is reducible if and only if $k=0$, in which case it decomposes into two irreducible constituents. Moreover, \(\V i^{\KS_{1}}\cong \V k^{\KS_{1}}\), exactly when $i+k= \n$.
	\item If $4| n$ then \(\V k^{\KS_{1}}\) is reducible  if and only if $k=0$ or $k=\frac{n}{4}$. In which case, it decomposes into two irreducible constituents Moreover, \(\V i^{\KS_{1}}\cong \V k^{\KS_{1}}\), exactly when $i+k= \n$.
\end{enumerate}	
\end{theo}

\begin{proof}
	The decomposition and irreducibility results follow from Corollary~\ref{mainresultsl2}. The multiplicity results are shown in Corollary~\ref{multiplicitysl}, and the fact that $\co W_{0,l}^{+}$ and $\co W_{0,l}^{-}$ have the same degree follows from Corollary~\ref{Wpmsamedimension}.
\end{proof}

\bibliographystyle{abbrv}
\bibliography{K-type-paper}
\end{document}